\documentclass[12pt]{amsart}
\usepackage{geometry}                
\usepackage{graphicx, color}
\usepackage{amssymb, amsmath}
\usepackage{epstopdf}
\newtheorem{theorem}{Theorem}[section]
\newtheorem{lemma}[theorem]{Lemma}
\newtheorem{corollary}[theorem]{Corollary}
\newtheorem{proposition}[theorem]{Proposition}

\newtheorem{definition}[theorem]{Definition}
\newtheorem*{uthm}{Theorem}

\newcommand{\Z}{\mathbb{Z}}

\newcommand{\C}{\mathbb{C}}
\newcommand{\D}{\mathbb{D}}

\newcommand{\acts}{\curvearrowright}

\title{Measure free factors of free groups}
\author{Juan Alonso}

\begin{document}
\maketitle

\begin{abstract} Measure free factors are a generalization of the notion of free factors of a group, in a measure theoretic context. We find new families of cyclic measure free factors of free groups and some virtually free groups, following a question by D. Gaboriau. 
\end{abstract}

\section{Introduction}

The notion of {\em measure equivalence} between groups was introduced by M. Gromov \cite{gromov}, as an analog of quasi-isometry in the context of actions on measure spaces. It has been widely studied since then, for example the works of Dye \cite{dye1}, \cite{dye2}, Ornstein and Weiss \cite{ow}, and Furman \cite{furman} characterized the measure class of $\Z$, found to be exactly the infinite amenable groups. This motivates the study of which groups are measure equivalent to the non-abelian free groups, $F_n$ with $n\geq 2$ (which are all virtually isomorphic, thus measure equivalent). 

In this paper we will work with the notion of {\em treeability}, in the sense of Pemantle and Peres \cite{pp}, which is equivalent to being measure equivalent to a free group, as shown by G. Hjorth \cite{hjorth}. 

The concept of {\em measure free factor} (definition \ref{mffdef}) was introduced by D. Gaboriau in \cite{gab-mff}, as a tool for the study of measure equivalence. He was able to find many groups that are measure equivalent to the free group $F_2$, by showing that this class is closed under certain amalgamated products. Namely, those amalgamations $A*_CB$ where $C$ is a measure free factor of either $A$ or $B$ (see \ref{amalgam} for the precise statement). In this work, Gaboriau \cite{gab-mff} poses the question of which elements of $F_2$ generate a cyclic measure free factor. He shows that such an element cannot be a proper power, and also finds the first non-trivial example (see \ref{surface1}). 

Gaboriau's results were then used by M. Bridson, M. Tweedale and H. Wilton \cite{bridson} to prove that hyperbolic {\em limit groups} (defined by Z. Sela \cite{sel1}) are measure equivalent to $F_2$. This provides a wide set of examples of groups in the measure equivelence class of $F_2$, and it naturally raises the question of whether all non-abelian limit groups are in this class. The importance of measure free factors in the study of this problem is explained in that same work \cite{bridson}, and it arises from the structure of limit groups as iterated amalgamations and HNN extensions, found by Z. Sela \cite{sel1} and O. Kharlampovich and A. Myasnikov \cite{km1}, \cite{km2}. Thus, new measure free factors of free or limit groups give rise to more limit groups in the measure class of $F_2$. 

This work advances the study of measure free factors of free groups, by finding some new infinite families of elements of $F_n$ that generate cyclic measure free factors of $F_n$. Namely, we prove the following.

\begin{uthm} {\rm \bf \ref{bswords}} Let $F=\langle x,y_1,\ldots,y_k \rangle$ be a free group of rank $k+1$. Then an element of the form \[ w=xy_1x^{m_1}y_1^{-1}y_2x^{m_2}y_2^{-1}\cdots y_kx^{m_k}y_k^{-1} \] generates a measure free factor of $F$.
\end{uthm}

\begin{uthm} {\rm \bf \ref{twoletters}} Let $G = F_2 = \langle a,b \rangle$. Then an element of the form $w=a^kb^n$ for $k,m \neq 0$ generates a measure free factor of $G$.
\end{uthm}

Since conjugates of measure free factors are also measure free factors, Theorem \ref{twoletters} gives that all the words of the form $a^kb^na^p$ with $n\neq 0, k\neq -p$ generate measure free factors of $F_2$. These are exactly all the {\em three-letter words} of $F_2$ that are not proper powers. In the special case of $F_2$, Theorem \ref{bswords} says that the words of the form $aba^mb^{-1}$ generate measure free factors of $F_2$.

This produces new examples of groups that are measure equivalent to a free group, for instance the limit groups $F_n*_wF_n$ and $F_n*_w\left(\langle w\rangle\times \Z^m\right)$, where $w$ is one of the elements mentioned in Theorems \ref{bswords} or \ref{twoletters}.

We also find measure free factors of some virtually free groups.

\begin{uthm} {\rm \bf \ref{vfree}} Let $G = \langle a_1,\ldots,a_n,s_1,\ldots,s_k | s_1^{n_1}=1,\ldots, s_k^{n_k}=1 \rangle \cong F_n*\Z_{n_1}*\cdots *\Z_{n_k}$. If $v\in F_n$ generates a measure free factor of $F_n$, then $w=vs_1^{p_1}\cdots s_k^{p_k}$ generates a measure free factor of $G$ for any $p_1,\ldots,p_k$.
\end{uthm}

The measure free factors obtained by Gaboriau \cite{gab-mff} are the boundary subgroups of certain surface groups (orientable with positive genus, see \ref{surface1}). Theorem \ref{vfree} allows us to generalize this to boundary subgroups of some 2-orbifold groups (with positive genus, see Corollary \ref{orbifold}).

The paper is organized as follows. Section \ref{sec:prelim} covers the preliminary notions and known results. First we give a brief introduction to Borel equivalence relations, treeings and cost, in subsection \ref{sec:borel}. Then, in subsection \ref{sec:mff}, we recall the definitions and results of Gaboriau \cite{gab-mff} that concern measure free factors. We also define {\em common measure free factors} (definition \ref{commonmff}), which are necessary for Proposition \ref{HNN}, a version for HNN extensions of Gaboriau's theorem about amalgamations (Theorem \ref{amalgam}). Subsection \ref{sec:ind} recalls the notions of induced and coinduced actions, and contains the proof of Proposition \ref{HNN}. Section \ref{sec:lift} is devoted to the main technical tool, Theorem \ref{lift}, which gives a way of passing to finite index subgroups in the problem of showing that some cyclic subgroup is a measure free factor. This result is very related to the Kurosh theorem for Borel equivalence relations obtained by A. Alvarez \cite{aurelien}, but going in the converse direction. Theorem \ref{lift} is then used in section \ref{sec:surfaces} to generalize the theorems of Gaboriau about the boundary subgroup of surface groups (\ref{surface1} and \ref{surface2}) to non-orientable surfaces (Lemma \ref{surface3}). We also prove a version for slightly more general systems of disjoint simple closed curves instead of the boundary subgroup (Proposition \ref{surface4}). Finally, section \ref{sec:mff-new} contains the proofs of Theorems \ref{bswords} and \ref{twoletters}. Section \ref{sec:mff-vfree} gives the proof of Theorem \ref{vfree}, as well as its corollary \ref{orbifold} which generalizes Proposition \ref{surface4} to the case of 2-orbifolds.     

I would like to thank D. Gaboriau for useful discussion and suggestions.

\section{Preliminaries}\label{sec:prelim}

\subsection{Treeings and cost}\label{sec:borel}


Here we review some of the theory of {\em Borel equivalence relations} on {\em standard Borel spaces}, and the notion of {\em cost}, an invariant of an equivalence relation preserving a probability measure. This material is covered with detail in the book by Kechris and Miller \cite{kechris}.

A {\em standard Borel space} is a measurable space $X$ (a set $X$, together with a $\sigma$-algebra of subsets of $X$), which is isomorphic as a measurable space to a Borel subset of the interval (with the $\sigma$-algebra of Borel sets). When $X$ is a standard Borel space, we will refer to the subsets in its $\sigma$-algebra as the Borel subsets of $X$. We will also call an isomorphism of measurable spaces between standard Borel spaces  a  {\em Borel isomorphism}.

It is a well known result that $X$ is a standard Borel space if and only if it is isomorphic to a Borel subset of any Polish space (see for instance \cite{kechris-classic}). Thus products of standard Borel spaces are again standard Borel spaces.

A {\em Borel equivalence relation} on the standard Borel space $X$ is an equivalence relation $E\subset X\times X$ which is a Borel subset of $X\times X$.

Borel equivalence relations are one of the basic objects of this work. The main example we will consider is when a group $G$ acts on $X$ by Borel automorphisms. Then the equivalence relation $E_G^X$, whose classes are the $G$-orbits, is Borel. It is a fact that every Borel equivalence relation with countable classes comes from the action of some group (Feldman--Moore Theorem, \cite[Theorem 15.1]{kechris} ), but we will not need to make use of this.


Now we turn to {\em graphings}, which play the role of generators for these equivalence relations.
 
Let $X$ be a standard Borel space. A {\em partially defined isomorphism} on $X$ is a map $\varphi:A\to B$, where $A$ and $B$ are Borel subsets of $X$ and $\varphi$ is a Borel isomorphism between them. If $\Phi = \{\varphi_i:A_i \to B_i \}_{i\in I}$ is a family of partial isomorphisms, the equivalence relation {\em generated} by $\Phi$ is the minimal Borel equivalence relation $E\subset X\times X$ containing the set  \[ \{(x,\varphi_i(x)): x\in A_i, i\in I \} \]
we will denote it by $E_{\Phi}$.

\begin{definition} A {\em graphing} for $E$ is a family $\Phi = \{\varphi_i \}_{i\in I}$ of partially defined Borel isomorphisms on $X$ that generates $E$ 
\end{definition} 

As the name suggests, a graphing defines a graph structure on each equivalence class. Consider the graph with vertex set $X$ and edges $(x,\varphi_i(x))$ for $x\in A_i$, $i\in I$. Then its connected components are the equivalence classes of $E$, and this gives the graph structure of each class. 

As an example, let $E=E_G^X$ be the orbit relation induced by an action $G \acts X$. If $S=\{ s_i\}_{i\in I}$ is a generating set for $G$, then the maps $\varphi_i:X\to X$ s.t. $\varphi_i(x) = s_i\cdot x$ form a graphing for $E_G^X$. In this case $A_i=B_i=X$ for all $i$. The graph structure is the same for every orbit, and agrees with the Cayley graph of $(G,S)$.


Now let $\mu$ be a finite measure on the standard Borel space $X$. We say that the equivalence relation $E$ {\em preserves} $\mu$, or is {\em measure preserving} (m.p.), if it admits a graphing $\Phi = \{\varphi_i:A_i \to B_i \}$ in which every transformation $\varphi_i$ preserves $\mu$, i.e. $(\varphi_i)_* \mu|_{A_i} = \mu|_{B_i}$. It is easy to check that {\em every} graphing of a m.p. equivalence relation satisfies that property.

If $G \acts X$ is a Borel action, the orbit relation $E_G^X$ preserves $\mu$ if and only if $G$ acts by measure preserving Borel automorphisms of $X$. In this case we say that the action is m.p.. When $\mu$ is a probability measure, we call it {\em probability measure preserving} (p.m.p.).

In the context of m.p. Borel actions, we want to relax the definition of a free action, to mean free almost everywhere.

\begin{definition} A m.p. Borel action $G\acts X$ is called {\em free} if the set of $x\in X$ such that $g\cdot x= x$ for a non-trivial $g\in G$ has measure zero.
\end{definition}

The equivalence relations we will consider in the rest of the paper will be the ones induced by free p.m.p. Borel actions, and their sub-relations.

For a m.p. equivalence relation, we can define its {\em cost}, an invariant introduced by G. Levitt \cite{levitt}, and studied extensively by D. Gaboriau \cite{gab-mercuriale}, \cite{gab-cout}. In the analogy between graphings and group generators, the cost would correspond to the rank.

\begin{definition} Let $X$ be a standard Borel space with a finite measure $\mu$. 
\begin{enumerate}
\item Let $\Phi = \{\varphi_i:A_i \to B_i \}$ be a measure preserving graphing. The {\em cost} of $\Phi$ is \[C_{\mu}(\Phi) = \sum_i \mu(A_i) \] 
\item Let $E$ be a Borel m.p. equivalence relation on $X$. The cost of $E$ is the infimum of the costs of its graphings, i.e.
\[ C_{\mu}(E) = \inf \{ C(\Phi) : \Phi \mbox{ is a graphing for } E \} \]
\end{enumerate}
\end{definition}

It is clear that $C_{\lambda\mu} = \lambda C_{\mu}$ both for graphings and equivalence relations. When $\mu$ is a probability, we drop it from the notation, writing $C$ for $C_{\mu}$.  

The cost also provides an invariant for groups.

\begin{definition} The {\em cost} of a group $G$ is the infimum of the costs $C(E_G^X)$ over all free Borel p.m.p. actions $G\acts X$ on standard Borel probability spaces.
\end{definition}

The most important kind of graphings are {\em treeings}, which we define now. Treeings are the analog of free bases for a free group, in the context of m.p. Borel equivalence relations.

\begin{definition} Let $E$ be a m.p. Borel equivalence, and $\Phi$ be a graphing for $E$. We say that $\Phi$ is a {\em treeing} if the graph induced by $\Phi$ on each class of $E$ is a tree for almost every class.
\end{definition}

Not every m.p. Borel equivalence relation admits a treeing. If it does, it is called {\em treeable}. As an example, consider the relation $E_G^X$ induced by a free m.p. Borel action $G\acts X$, and the graphing $\Phi$ given by a generating set $S$ of $G$. In this case, $\Phi$ is a treeing if and only if $G$ is a free group and $S$ is a free basis. If $G$ is not free, it is still possible that $E_G^X$ may admit a treeing for some free m.p. Borel action. However, to determine which groups do so is an open problem. 

\begin{definition} A group $G$ is {\em treeable} if there exists a free p.m.p. Borel action $G\acts X$ such that $E_G^X$ is treeable. 
\end{definition}

The following theorems, due to Gaboriau, are the fundamental results in this theory which we will employ throughout this paper.


\begin{theorem}\label{subtree} \cite[Theorem 5]{gab-cout} Let $E$ be a m.p. Borel equivalence relation with countable classes. If $E$ is treeable and $F\subset E$ is a sub-equivalence relation, then $F$ is also treeable.
\end{theorem}

Thus a subgroup of a treeable group is also treeable. There is a close relationship between treeings and cost, given by the following theorems.


\begin{theorem} \label{tree0} \cite[Theorem 1]{gab-cout} Let $E$ be a m.p. Borel equivalence relation on the standard Borel space $X$. If $\Phi$ is a treeing for $E$ then $C_{\mu}(\Phi) = C_{\mu}(E)$. 
\end{theorem}

If the cost is finite, the converse holds.

\begin{theorem} \label{tree1} \cite[Theorem 1, Proposition I.11]{gab-cout} Let $E$ be a m.p. Borel equivalence relation on the standard Borel space $X$, with $C_{\mu}(E)<\infty$. A graphing $\Phi$ for $E$ is a treeing if and only if $C_{\mu}(\Phi)=C_{\mu}(E)$.
\end{theorem}

\begin{theorem} \label{tree2} \cite[Proposition 30.5]{kechris} Let $G\acts X$ a free p.m.p. Borel action. If $E_G^X$ is treeable then $C(G)=C(E_G^X)$.
\end{theorem}


Finally, we recall {\em complete sections}, which are often useful to compute cost.

\begin{definition} Let $E$ be a Borel equivalence relation on the standard Borel space $X$. A {\em complete section} for $E$ is a Borel subset $A\subset X$ meeting every class of $E$.
\end{definition}

If $E$ has countable classes and is m.p., then a complete section always has positive measure. By disregarding sets of measure zero, we need only ask that a complete section meets almost every class. If $E$ is an equivalence relation on $X$, and $A\subset X$, one defines the {\em restriction} of $E$ to $A$ as follows.

\[ E|_A = \{ (x,y) \in A\times A : (x,y)\in E \} \]

The relationship between cost and complete sections is given by the formula below.

\begin{theorem} \cite[Proposition II.6]{gab-cout} Let $E$ be a m.p. Borel equivalence relation with countable classes, on the standard Borel measure space $X$. Let $A\subset X$ be a complete section for $E$. Then
\[ C_{\mu}(E) = C_{\mu|_A}(E|_A) + \mu(X\setminus A) \]
\end{theorem}

\subsection{Measure free factors}\label{sec:mff}

Here we discuss the notion of {\em measure free factors} of groups, which is the main subject of our study. Then we will turn to the problem of finding such measure free factors, recalling the results of D. Gaboriau \cite{gab-mff}. First we need the following definitions about equivalence relations.

\begin{definition} Let $E_1$, $E_2$ be Borel equivalence relations on the standard Borel space $X$. We say that $E_1$ and $E_2$ are {\em orthogonal}, and write $E_1\perp E_2$, if for every cycle $(x_i)$, $i\in \Z_{2n}$, of elements of $X$ such that 
\begin{enumerate}
\item $(x_i,x_{i+1}) \in E_1$ for all $i$ odd. 
\item $(x_i,x_{i+1}) \in E_2$ for all $i$ even. 
\end{enumerate}
we have that $x_i=x_{i+1}$ for some $i$.
\end{definition}

\begin{definition}  Let $E$ be a m.p. Borel equivalence relation. We say that $E$ is the {\em free product} of the Borel sub-equivalence relations $E_1$, $E_2$, and write $E=E_1*E_2$, if
\begin{enumerate}
\item $E$ is generated by their union $E_1\cup E_2$.
\item There is a full measure set $B$ so that $E_1|_B\perp E_2|_B$. (i.e. $E_1\perp E_2$ for almost every class of $E$).
\end{enumerate} 
\end{definition} 

The case for multiple factors is a clear generalization.

These definitions reflect the notion of free product for groups. Specifically, if $G\acts X$ is a free m.p. Borel action, and $G$ splits as  $G=H*K$, then it is easy to check that $E_G^X = E_{H}^X*E_K^X$.


\begin{definition} \label{mffdef} \cite[Definition 3.1]{gab-mff} A subgroup $H\leq G$ is a {\em measure free factor} of $G$ if there exists a free p.m.p. Borel action $G\acts X$ on a standard Borel probability space, and a Borel equivalence relation $E'$ on $X$ such that \[ E_G^X = E_H^X*E' \]
\end{definition}


From the remark above, it is clear that if $H$ is a free factor of $G$ then it is also a measure free factor. It is also an easy fact that the image of a measure free factor by an automorphism of $G$ is again a measure free factor of $G$.
Free factors are not the only examples of measure free factors, as shown by the following theorems of Gaboriau.


\begin{theorem} \label{surface1} \cite[Theorem 3.2]{gab-mff} Let $F = \langle a_1,b_1,\ldots,a_g,b_g \rangle$ be a free group of rank $2g$. Then the element $w=[a_1,b_1]\cdots [a_g,b_g]$ generates a measure free factor of $F$.
\end{theorem}

\begin{corollary} \label{surface2} \cite[Corollary 3.5]{gab-mff} Let $S$ be an orientable surface with boundary, of genus at least $1$. Let $\gamma_1,\ldots,\gamma_k$ be the boundary curves of $S$. Then the boundary subgroup $ \langle [\gamma_1],\ldots,[\gamma_k]\rangle \leq \pi_1(S)$ is a measure free factor of $\pi_1(S)$.
\end{corollary}

In general, finding measure free factors is hard. Gaboriau \cite{gab-mff} posed the problem of finding which elements of a free group generate a cyclic measure free factor. It is shown in \cite{gab-mff} that such an element cannot be a proper power.



Measure free factors can be used to construct new treeable groups, via amalgamated products. The following is obtained by combining the arguments of Theorem 3.13 and Theorem 3.17 in \cite{gab-mff}. 

\begin{theorem} \label{amalgam} Let $G=\Gamma_1*_{\Lambda}\Gamma_2$, where $\Gamma_i$ are treeable groups, and $\Lambda$ is a measure free factor of $\Gamma_1$. Let also $H\leq \Gamma_2$ be a measure free factor of $\Gamma_2$. Then $G$ is treeable, and $H\leq G$ is a measure free factor of $G$.
\end{theorem}

In particular, if $\Lambda$ is also a measure free factor of $\Gamma_2$, then $\Lambda$ is a measure free factor of the amalgamation $G$.


We would like to have a similar result for HNN extensions. Let $G$ be a group, $H\leq G$ a subgroup and $\alpha:H\to G$ an injective homomorphism. We define the HNN extension $G*_H$ by the following presentation: \[ G*_H = \langle G,t | tht^{-1} = \alpha(h) \mbox{ for } h\in H \rangle\] To conclude that $G*_H$ is treeable we need stronger hypotheses: We still assume that $H$ is a measure free factor of $G$, but also that $\alpha(H)$ is contained in a subgroup $H'\leq G$ so that $H$ and $H'$ are {\em common measure free factors} of $G$, as defined next.



\begin{definition}\label{commonmff} Let $G$ be a group, and $H,K\leq G$ subgroups. Then $H$ and $K$ are {\em common measure free factors} of $G$ if there exists a free p.m.p. Borel action $G\acts X$ on a standard Borel probability space, and a Borel equivalence relation $E'$ on $X$ such that \[ E_G^X = E_H^X*E_K^X*E' \] 
\end{definition}

For more than two subgroups the definition is analogous. Notice that if $H$ and $K$ are common measure free factors, then $H\cap K = \{1 \}$ unless $H=K$, and moreover, the subgroup generated by $H$ and $K$ in $G$ is isomorphic to $H*K$. With similar techniques as those involved in \ref{amalgam}, we can prove the following.

\begin{proposition}\label{HNN} Let $G$ be a treeable group, $H,H',K$ be subgroups of $G$, and $\alpha:H\to H'$ be an injective homomorphism. If $H$, $H'$ and $K$ are common measure free factors of $G$, then the HNN extension $\Gamma = G*_H$ is treeable and $K\leq\Gamma$ is a measure free factor of $\Gamma$.
\end{proposition}

We give a proof of this in the next section. Considering common measure free factors gives also a refinement of \ref{amalgam}, which is obtained by the same argument we will use for \ref{HNN}.


\begin{proposition} \label{amalgam+} Let $G=\Gamma_1*_{\Lambda}\Gamma_2$, where $\Gamma_i$ are treeable groups. Let $K\leq \Gamma_1$ and assume that $\Lambda$ and $K$ are common measure free factors of $\Gamma_1$. Also let $H\leq \Gamma_2$ be a measure free factor of $\Gamma_2$. Then $G$ is treeable, and $H$ and $K$ are common measure free factors of $G$.
\end{proposition}

\subsection{Induced and coinduced actions}\label{sec:ind}

In this section we present tools for extending an action of a subgroup to an action of the overgroup. Specifically, if we have groups $H\leq G$, and an action $H\acts X$, we wish to construct an action $G\acts Y$ that contains, in some sense, the action of $H$ on $X$. These tools are the {\em induced} and {\em coinduced} actions.

\begin{definition} Let $H\leq G$ be a subgroup and $H\acts X$ be an action. Define
\[ \mbox{\em CoInd}_H^G X = \{\psi:G \to X : \psi(gh^{-1}) = h\cdot \psi(g) \mbox{\ for } h\in H, g\in G \} \]
with the action of $G$ given by $g\cdot\psi(k) = \psi(g^{-1}k)$ for $g,k\in G$.
\end{definition}

The coinduced action satisfies the following general properties, which are easy to prove:

\begin{enumerate}
\item The map $p:\mbox{CoInd}_H^G X \to X$ taking $\psi$ to $\psi(1)$ is $H$-equivariant and surjective.
\item If $\{ g_i\}$ is a set of representatives of $G/H$, then the map \[ \mbox{CoInd}_H^G X \to X^{G/H}  \mbox{\ s.t. } \psi \to \{\psi(g_i)\}_{g_iH} \]
is a bijection.
\item Let $G$ act on $X^{G/H}$ as follows: if $g\in G$, $f \in X^{G/H}$ then put \[ (g\cdot f)_{g_iH} = hf_{g_jH} \mbox{\ where } g^{-1}g_i = g_jh^{-1} \mbox{\ for } h\in H \]
With respect to this action, the map defined in (2) is an equivariant isomorphism.
\item If $H\acts X$ is free, then so is $G\acts \mbox{CoInd}_H^G X$.
\item If $H\acts (X,\mu)$ is a p.m.p. Borel action, then $\mbox{CoInd}_H^G X$ can be given the product Borel structure and the product measure of $X^{G/H}$. The action $G\acts \mbox{CoInd}_H^G X$ is  Borel and p.m.p. 
\end{enumerate}

Some of the interesting properties of an action are preserved under coinduction, as the next lemma shows. It is the key for our applications of this construction.

\begin{lemma}\label{coind1} Let $H\leq G$, and $H\acts X$ be a free p.m.p. Borel action with a graphing $\Phi$ that generates $E_H^X$. Let $Y=\mbox{\em CoInd}_H^G X$. There exists a graphing $\hat \Phi$ that generates $E_H^Y$ and has $C(\hat \Phi) = C(\Phi)$. Moreover, if $\Phi$ is a treeing so is $\hat \Phi$.
\end{lemma}

{\em Proof:} Let $\Phi = \{\varphi_i:A_i \to B_i \}$. By subdividing the sets $A_i$, we can assume that $\varphi_i(x) = h_i\cdot x$ for $h_i\in H$ and all $x\in A_i$. Now define $\hat A_i = p^{-1}(A_i)$, $\hat B_i = p^{-1}(B_i)$ and $\hat \varphi_i(y)=h_i\cdot y$ for all $y\in\hat A_i$. Put $\hat \Phi = \{\hat \varphi_i \}$. Then $\hat \Phi$ generates $E_H^Y$ by equivariance of the map $p$ and freeness of the action $H\acts X$. Also, $C(\hat \Phi) = C(\Phi)$ by definition of the product measure. Finally, if $\Phi$ is a treeing, then $\hat \Phi$ must also be a treeing, since a non-trivial cycle in the graphing $\hat \Phi$ would project under $p$ to a non-trivial cycle of $\Phi$ in $X$.

$\Box$

\begin{lemma} \label{freeprod} Let $G=G_1*G_2$, and let $H_1, H_2$ be a measure free factors of $G_1, G_2$ respectively. Then $H_1$ and $H_2$ are common measure free factors of $G$.
\end{lemma}

{\em Proof:}
For $i=1,2$, let $G_i\acts X_i$ be free p.m.p. Borel actions such that $E_{G_i}^{X_i} = E_{H_i}^{X_i}*E'_i$. Take $Y_i=\mbox{CoInd}_{G_i}^GX_i$, and $Y=Y_1\times Y_2$ with the diagonal action. Then we get $E_{G_i}^{Y} = E_{H_i}^{Y}*E''_i$, where $E''_i$ is generated by $\hat \Psi_i$, the graphing obtained as in Lemma \ref{coind1} from a graphing $\Psi_i$ of $E'_i$. Since $G=G_1*G_2$ we have \[ E_G^Y = E_{G_1}^Y*E_{G_2}^Y = E_{H_1}^Y*E''_1*E_{H_2}^Y*E''_2 \] which gives the lemma.

$\Box$

{\bf Proof of Proposition \ref{HNN}:}
Recall the notations of the proposition: $H,H',K\leq G$ are common measure free factors, $\alpha:H\to H'$ is an isomorphism and $\Gamma = G*_H$ is the corresponding HNN extension. 

Let $G\acts X$ be a free p.m.p. Borel action such that $E_G^X$ is treeable and $E_G^X = E_H^X*E_{H'}^X*E_K^X*E'$. This can be done by taking a diagonal action $G\acts X=X_0\times X_1$ where $E_G^{X_0}$ is treeable and $E_G^{X_1}$ splits as the desired free product. 

Let $Y = \mbox{CoInd}_G^{\Gamma} X $. Take treeings $\Phi_H$, $\Phi_K$ and $\Psi$ for $E_H^X$, $E_K^X$ and $E'$ respectively, which exist by \ref{subtree}, since $E_G^X$ is treeable. Consider $\hat \Phi_H$, $\hat \Phi_K$ and $\hat \Psi$ as in the previous Lemma \ref{coind1}, which are still treeings.  By the same argument as in \ref{coind1} (projecting by $p$), we can obtain that $E_G^Y=E_H^Y*E_{H'}^Y*E_K^Y*E''$ where $E''$ is the equivalence relation generated by $\hat \Psi$. Finally, consider the Borel automorphism of $Y$ given by the action of the stable letter $t$, which we still denote by $t$, i.e. $t:Y \to Y$ s.t. $t(y)=t\cdot y$.

Put $\Omega = \{ t\}\cup \hat \Phi_{H'} \cup \hat \Phi_K \cup \hat \Psi$, which is a graphing on $Y$. We claim that $\Omega$ is a treeing for $E_{\Gamma}^Y$.

It generates $E_{\Gamma}^Y$: It clearly generates $E_{H'}^Y$, $E_K^Y$ and $E''$. Also, since $E_H^Y = t^{-1}E_{H'}^Y t$, we see that $\Omega$ generates $E_G^Y$. Since $\Gamma$ is generated by $G$ and $t$, it is now clear that $\Omega$ generates $E_{\Gamma}^Y$.

There are no non-trivial cycles a.e.: Suppose $\phi_{i_1}\cdots \phi_{i_n}(y)=y$ for $y$ in a set of positive measure, $\phi_j\in\Omega$, so that $\phi_{i_1}\cdots \phi_{i_n}$ is a reduced word on $\Omega$. This gives rise to a reduced word $\gamma_{i_1}\cdots \gamma_{i_n}$ with letters in $G\cup\{t, t^{-1}\}$ representing $1$ in $\Gamma$. If all of its letters are in $G$, we get a contradiction, since $\hat \Phi_{H'} \cup \hat \Phi_K \cup \hat \Psi$ is a treeing of a subrelation of $E_G^Y$ (each one is a treeing, and they are mutually orthogonal since $E_G^Y=E_H^Y*E_{H'}^Y*E_K^Y*E''$). If some $\gamma_i$ equals $t$ or $t^{-1}$, that contradicts the normal form for an HNN extension, since the treeing $\hat \Phi_{H'} \cup \hat \Phi_K \cup \hat \Psi$ generates an equivalence relation which is orthogonal to $E_H^Y$, thus no element of $H$ appears in the normal form of $\gamma_{i_1}\cdots \gamma_{i_n}$. 

$\Box$


Now we turn to the induced action. In our context, it is only useful for extensions of finite index. However, it will be the main tool for the next section.

\begin{definition} Let $H\leq G$ be a subgroup and $H\acts X$ be an action. Define \[ \mbox{\em Ind}_H^G X = (X\times G)/H \] where the quotient is by the right diagonal action of $H$ ($h\cdot(x,g) = (h^{-1}\cdot x,gh)$). Let $G$ act on $X\times G$ by left multiplication on the second coordinate. This induces the action of $G$ on $\mbox{\em Ind}_H^G X$.
\end{definition}

These are the main properties of the induced action, whose proofs are straightforward:

\begin{enumerate}
\item $\mbox{Ind}_H^G X$ can be identified with $X\times G/H$.
\item If $\{ g_i\}$ is a set of representatives of $G/H$, then we can write $\mbox{Ind}_H^G X = X\times G/H =\bigcup_i(X\times\{g_iH\})$, which is a union of disjoint copies of $X$.
\item The inclussion $X \to X\times\{H\} \subset \mbox{Ind}_H^G X$ is $H$-equivariant. We call $X_0 = X \times \{H\}$.
\item If $H\acts X$ is p.m.p. and the index of $H$ in $G$ is finite, then the union measure (rescaled by the index of $H$ in $G$) is an invariant probability measure on $\mbox{Ind}_H^G X$.
\item $X_0$ is a complete section for the orbits of $G$ on $\mbox{Ind}_H^G X$. Its translates are of the form $X\times\{g_iH\} = g_iX_0$.
\item The restriction to $X_0$ of the orbit equivalence of $G$ on $\mbox{Ind}_H^G X$ is the orbit equivalence of $H$ on $X_0 =X$. In symbols \[ E_G^{\mbox{Ind}_H^G X}|_{X_0} = E_H^{X_0} \cong E_H^X \]
\end{enumerate}

\section{Lifting to finite covers}\label{sec:lift}

In this section we prove Theorem \ref{lift}, which is the main tool in the proofs of our results. Our general goal is to show that some element $w\in G$ of a treeable group $G$ generates a measure free factor of $G$. This result will allow us to pass to a finite index subgroup $H$ of $G$, replacing $\langle w\rangle$ by a suitable subgroup of $H$. First we explain what this suitable subgroup is, starting from the geometric viewpoint.

Let $G$ be a group and $H\leq G$ a subgroup of finite index $n=[G:H]$. Consider a complex $X_G$ with $\pi_1(X_G) = G$, and the $n$-sheeted covering space $X_H \to X_G$ corresponding to $H\leq G$. Recall that the conjugacy classes in $G$ correspond to the homotopy classes of closed curves in $X_G$.

Let $w\in G$ and take a closed curve $\gamma$ in $X_G$ that represents the conjugacy class of $w$. The pre-image of $\gamma$ in $X_H$ is a union of closed curves $\gamma_1,\ldots,\gamma_k$, where $\gamma_i$ is the union of $m_i$ lifts of $\gamma$. In that sense, $\gamma_i$ ``covers" $\gamma$ with index $m_i$. 
Each of these curves $\gamma_i$ defines a conjugacy class $[w_i]$ in $H$, where we can write $w_i = g_i^{-1}w^{m_i}g_i$ for some $g_i\in G$. 

More explicitly, let $p_0\in X_G$, $p\in X_H$ be the basepoints (with $p$ projecting to $p_0$), and let $p_i$ be a pre-image of $p_0$ so that $\gamma_i$ can be obtained as the lift of $\gamma^{m_i}$ starting at $p_i$. Take a curve $\alpha_i$ in $X_H$ going from $p$ to $p_i$, and let $g_i^{-1}\in G=\pi_1(X_G,p_0)$ be the homotopy class of the projection of $\alpha_i$. Then we obtain $[\alpha_i \gamma_i \alpha_i^{-1}] = g_i^{-1}w^{m_i}g_i$. 

Notice that the choice of $\alpha_i$ corresponds to the choice of the representative $g_i$ in the coset $g_iH$, which also corresponds to the choice of $w_i$ as representative of its conjugacy class in $H$. On the other hand, the choice of $p_i$ corresponds to the choice of $g_i$ in the right coset $\langle w \rangle g_i$. Thus the $g_i$ form a set of representatives of the double cosets in $\langle w\rangle \backslash G/H$.

This motivates the following definitions.
 
\begin{definition} Suppose $H\leq G$ is a subgroup of finite index, and $w\in G$. 
\begin{itemize}
\item[(a)] If $g\in G$, let $m(g)$ be the minimum $t$ such that $g^{-1}w^tg\in H$. We will say that the element $g^{-1}w^{m(g)}g$ is the {\em lift} of $w$ to $H$ with respect to $g$.
\item[(b)] Let $\{g_1,\ldots,g_k\}$ be a set of representatives of the double cosets in $\langle w\rangle \backslash G/H$. Put $m_i=m(g_i)$, and $w_i = g_i^{-1}w^{m_i}g_i$. Then we say that the set $\{w_1,\ldots,w_k\}$ is a {\em complete lift} of $w$ to $H$.
\end{itemize}
\end{definition}

This definition of lifts follows the one used by J. Manning (Definition 1.4 in \cite{manning}), with the difference that here we have different complete lifts for the various choices of double coset representatives, instead of defining it as a set of conjugacy classes. It bears this same relationship with the more general notion of {\em elevations}, introduced by D. Wise \cite{wise}.

\begin{theorem}\label{lift} Let $G$ be a treeable group of finite cost. Let $w\in G$, and $H\leq G$ be a subgroup of finite index $n=[G:H]$. Take   $\{g_1,\ldots,g_k\}$, a set of representatives of the double cosets in $\langle w\rangle \backslash G/H$, and let $K=\langle w_1,\ldots,w_k \rangle \leq H$ be the subgroup generated by the corresponding complete lift of $w$ to $H$. Assume that
\begin{enumerate}
\item $K$ is free of rank $k$, i.e. $w_1,\ldots,w_k$ is a free basis of $K$.
\item $K$ is a measure free factor of $H$.
\end{enumerate}
Then $\langle w \rangle$ is a measure free factor of $G$
\end{theorem}

{\em Proof:}
Let $H\acts X$ be a free p.m.p. Borel action that realizes $K$ as a measure free factor of $H$, i.e. $E_H^X = E_K^X*E'$, and that is treeable. This is possible by taking the direct product of actions satisfying each condition. Consider the induced action of $G$ on $Y=\mbox{Ind}_H^G X = X\times G/H$, and put $X_i = X\times \{g_iH \}$. Assume $g_1\in H$, so $X_1=X_0$ is the standard embedding of $X$ into $\mbox{Ind}_H^G X$. 
 
We can also assume that $g_1=1$: replacing $g_i$ by $g_ig_1^{-1}$ changes $K$ into $g_1Kg_1^{-1}$, which satisfies the same hypotheses.

Define $\phi_i:X_0 \to X_i$ by $\phi_i(x) = g_i\cdot x$. Also take $\Phi'$ a treeing for $E'$, which exists since $E_H^X$ is treeable (here we identify $X_0=X$). Consider the graphing $\Phi = \{w,\phi_2,\ldots,\phi_k\}\cup \Phi'$. We will show that $\Phi$ is a treeing of $E_G^Y$, which implies the result.

To see that it generates $E_G^Y$, recall that $X_0$ is a complete section, and notice that every translate of it is of the form $w^tg_iX_0$ (since the $g_i$ are a set of representatives of the double cosets in $\langle w\rangle \backslash G/H$, then the $w^tg_iH$ cover all the cosets of $G/H$). We get that $w^tg_iX_0 = w^t X_i = w^t \phi_i(X_0)$, so we only need to show that $\Phi$ generates the restriction $E_G^Y|_{X_0} = E_H^{X_0}$. Take $x\in X_0$. Clearly every element of $\Phi'$ is defined on $x$. On the other hand \[ w_i\cdot x = g_i^{-1}w^{m_i}g_i\cdot x = \phi_i^{-1}w^{m_i}\phi_i(x) \] which is defined, since $x\in X_0$ and $w^{m_i}g_i\cdot x \in X_i$ (noticing that $X_i = g_iX_0 = g_ig_i^{-1}w^{m_i}g_iX_0 = w^{m_i}g_i X_0$).
Since $\Phi'$ and $w_1,\ldots,w_n$ generate $E_H^X$, then $\Phi$ generates $E_H^{X_0}$. So $\Phi$ generates $E_G^Y$.

Now, on one hand \[ C(E_G^Y) = (n-1)\mu(X) + C(E_G^Y|_{X_0}) = (n-1)\mu(X) + C(E_H^X) \] by the formula for a complete section, so \[ C(E_G^Y) = (n-1)\mu(X) + C(E') + k\mu(X) = (n+k-1)\mu(X) + C(E') \] since $K$ is free of rank $k$. On the other hand 
\[ C(\Phi) = \mu(Y) + (k-1)\mu(X) + C(\Phi') = (n+k-1)\mu(X) + C(E') \] Thus, by Theorem \ref{tree1}, $\Phi$ is a treeing of $E_G^Y$.

$\Box$ 

\section{Curves on surfaces} \label{sec:surfaces}

Here we extend the results of D. Gaboriau on measure free factors of surface groups.

\begin{lemma} \label{surface3} Let $S$ be a surface with boundary, and $\gamma_1,\ldots,\gamma_k$ be its boundary curves. If the genus of $S$ is at least $1$, then the boundary subgroup $ \langle [\gamma_1],\ldots,[\gamma_k]\rangle \leq \pi_1(S)$ is a measure free factor of $\pi_1(S)$.
\end{lemma}

{\em Proof:}
When $S$ is orientable, this is Gaboriau's Theorem \ref{surface2}. If $S$ is non-orientable with one boundary component, take the orientable double cover $\hat S$. By the former case, the boundary components of $\hat S$ generate a measure free factor of $\pi_1(\hat S)$. On the other hand, the boundary curves of $\hat S$ are the lifts of the boundary curve of $S$ to this two-sheeted cover. So Theorem \ref{lift} applies.

In case there is more than one boundary component, write \[ \pi_1(S) = \langle a_1,\ldots,a_g,c_1\ldots,c_{k-1} \rangle\] where the boundary classes are $[\gamma_1] = a_1^2\cdots a_g^2 c_1\cdots c_{k-1}$, and $[\gamma_j] = c_{j-1}$ for $j>1$. We just showed that $a_1^2\cdots a_g^2$ generates a measure free factor of the subgroup $\langle a_1,\ldots,a_g \rangle$. Since \[ \pi_1(S) = \langle a_1,\ldots,a_g \rangle * \langle c_1\ldots,c_{k-1} \rangle \] we get that the boundary subgroup $\langle a_1^2\cdots a_g^2,c_1,\ldots,c_{k-1} \rangle$ is a measure free factor.
$\Box$

\begin{proposition} \label{surface4} Let $S$ be a surface with boundary, and $\gamma_1,\ldots,\gamma_k$ be its boundary curves. Suppose that $\alpha = \{ \alpha_1,\ldots,\alpha_n \}$ is a family of disjoint essential simple closed curves on $S$, and $S_1,\ldots,S_t$ are the components of $S$ cut along $\alpha$. If $S_j$ has genus at least $1$ for every $j$, then the subgroup \[ \langle [\gamma_1],\ldots,[\gamma_k],[\alpha_1],\ldots,[\alpha_n]\rangle \leq \pi_1(S) \] is a measure free factor of $\pi_1(S)$. 

\end{proposition}

{\em Proof:}
Consider the graph of groups $\Gamma$ induced by cutting $S$ along $\alpha$. For simplicity, first consider the case where all the curves in $\alpha$ are two sided. Then there is one vertex for each component $S_j$, and one edge for each curve $\alpha_i$. The vertex groups are $\pi_1(S_j)$, and the edge groups are $\langle [\alpha_i] \rangle$. Let $\alpha_i^{\pm}$ be the sides of a tubular neighborhood of $\alpha_i$. Then $[\alpha_i^{\pm}]$ are the images of $[\alpha_i]$ in the adjacent vertex groups. 
Since each $S_j$ has genus at least one,  Lemma \ref{surface3} gives that the boundary subgroup of $\pi_1(S_j)$ is a measure free factor of $\pi_1(S_j)$. This boundary subgroup is freely generated by the $[\alpha_i^{\pm}]$ and the $[\gamma_p]$ that lie in $S_j$.
Now $\pi_1(S)$ is obtained as an iteration of amalgamated products and HNN extensions of the vertex groups $\pi_1(S)$. Each such extension gives the identification \[ [\alpha_i] = [\alpha_i^-] \mbox{\ \ \ } [\alpha_i] = t_i[\alpha_i^+]t_i^{-1} \] where $t_i=1$ in the case of an amalgamation, and is a new generator in the case of an HNN extension. Then Proposition \ref{HNN} gives the proposition.

For the general case, let $\beta=\{\alpha_1,\ldots,\alpha_m\}$ be the curves of $\alpha$ that are cores of M\"obius bands in $S$. Cut $S$ along $\beta$, forming $S|_{\beta}$, and apply the former case to it. Notice that $[\alpha_i^2]$ for $i\leq m$ are boundary curves of $S|_{\beta}$. On the other hand, $\pi_1(S)$ is obtained from $\pi_1(S|_{\beta})$ by amalgamating $\langle [\alpha_i]\rangle$ along $[\alpha_i^2]$ for $i\leq m$. Then the result is obtained by \ref{amalgam+}.

$\Box$



\section{Cyclic measure free factors of $F_n$} \label{sec:mff-new}

In this section we apply the  results of the last section to find some cyclic measure free factors of free groups.

\begin{theorem} \label{bswords} Let $F=\langle x,y_1,\ldots,y_k \rangle$ be a free group of rank $k+1$. Then an element of the form \[ w=xy_1x^{m_1}y_1^{-1}y_2x^{m_2}y_2^{-1}\cdots y_kx^{m_k}y_k^{-1} \] generates a measure free factor of $F$.
\end{theorem}

{\em Proof:}
Let $H=\langle x,c_1,\ldots,c_k\rangle$ be another free group on $k+1$ generators, and put $v=xc_1\cdots c_k$. Then we can obtain $F$ as an HNN extension: \[ F = \langle H,y_1,\ldots,y_k | c_j = y_jx^{m_j}y_j^{-1} \mbox{, for } j=1,\ldots,k \rangle \]
The standard inclussion $H \hookrightarrow F$ maps $v$ to $w$, so we may regard $w$ as an element of $H$.
A natural complex with fundamental group $H$ is a $k+2$-punctured sphere $S$, whose boundary components represent $x,c_1,\ldots,c_k$ and $w=v=xc_1\cdots c_k$. We will identify these boundary curves with their representatives in $H=\pi_1(S)$. Starting from $S$, we can build a complex with fundamental group $F$ by attaching cylinders $C_i$, glued to $S$ by their boundary curves. More explicitely, one of the boundary components of $C_i$ is identified with $c_i$, and the other is glued to $x$ by an attaching map of degree $m_i$, so it represents $x^{m_i}$ in the fundamental group. We call this complex $X$. Notice that $w$ is the only boundary component of $S$ that was not attached to a cylinder in $X$.

Next we construct a finite cover $\hat S \to S$, proceeding as follows: Let $p>0$ be an integer with $(p,m_j)=1$ for all $j=1,\ldots,k$ and $(p,k+1)=1$ (e.g. $p$ prime, large enough). Let $\hat H$ be the kernel of the morphism $H\to \Z_p$ sending the generators $x,c_1,\ldots,c_k$ to $1$, and let $\hat S$ be the cover corresponding to $\hat H$. Notice that it has index $p$, and its boundary components are represented by $\hat x=x^p$, $\hat c_j = c_j^p$ and $\hat w = w^p$. This is because the images of $x$, $c_j$ and $w$ in $\Z_p$ have order $p$. By these reasons we obtain that $\hat w =w^p$ is a complete lift of $w$ to $\hat H$. 
Now, $\hat S$ is a surface of positive genus, so we may apply Lemma \ref{surface3} to conclude that its boundary subgroup (which is freely generated by $\hat x,\hat c_1,\ldots,\hat c_k,\hat w$), is a measure free factor of $\hat H$.

On the other hand, we will extend $\hat S$ to a $p$-sheeted cover $\hat X \to X$. Consider the standard $p$-sheeted cover $\hat C_i \to C_i$ of each cylinder $C_i$, and glue each of its boundary components to $\hat S$ along the curves $\hat c_i$ and $\hat x^{m_i}$ respectively. Since $(p,m_i)=1$, the covering maps $\hat S \to S$ and $\hat C_i \to C_i$ agree on the glued boundaries, so they give rise to a covering map $\hat X \to X$.

\begin{figure}[htbp]
\centering
\includegraphics[width=0.75\textwidth]{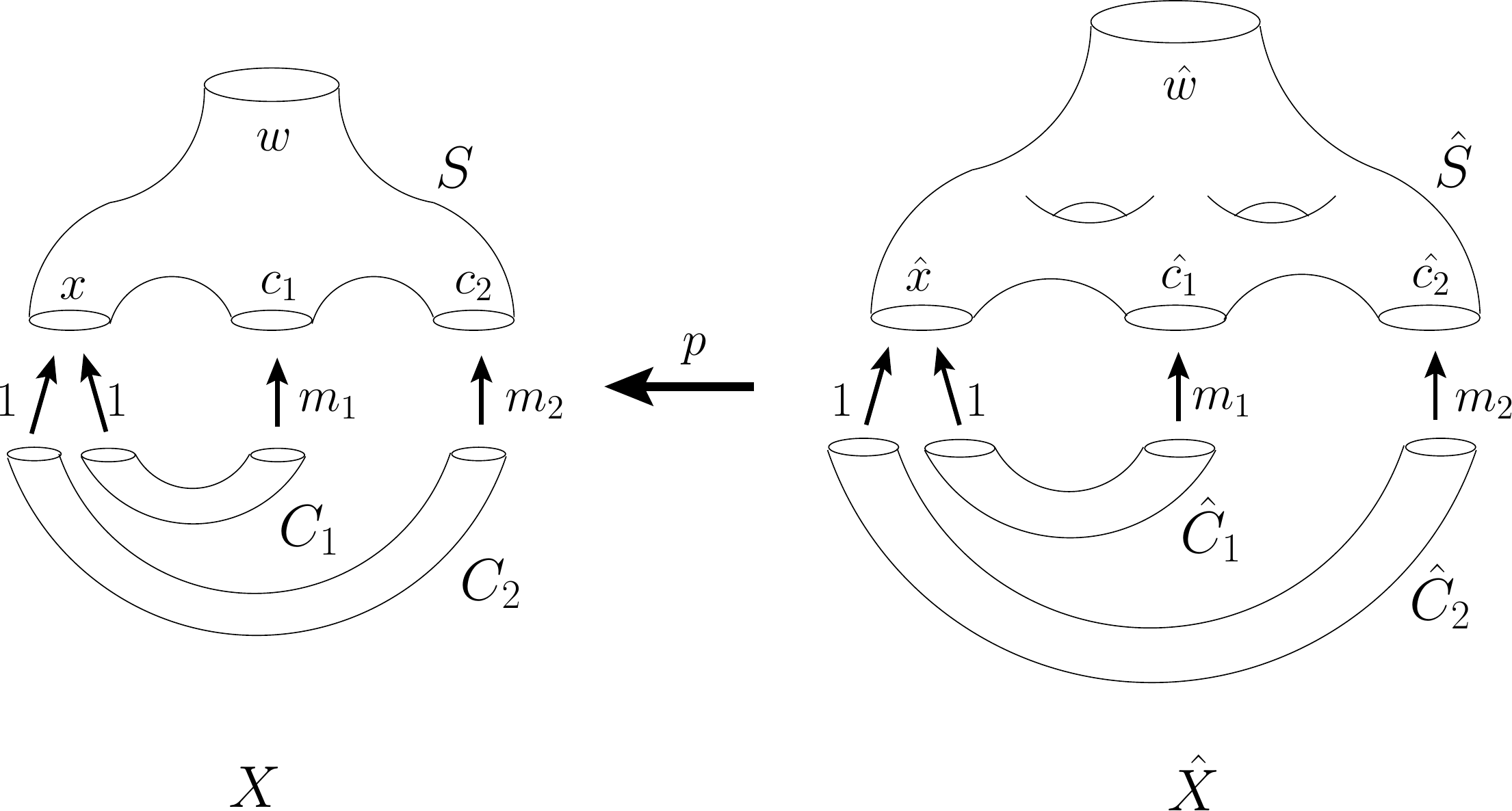}
\caption{{\bf Left:} Sketch of the complex $X$ for $k=2$. The gluing maps are labeled by their degrees. {\bf Right:} The $p$-sheeted cover $\hat X$, sketched in the same fashion.}
\end{figure}

Let $\hat F$ be the index $p$ subgroup of $F$ corresponding to the cover $\hat X$. Then $\hat F \cap H = \hat H$, and $\hat w =w^p$ is also a complete lift of $w$ to $\hat F$. By our construction of $\hat X$, we can write $\hat F$ as an HNN extension: 

\[ \hat F = \langle \hat H,l_1,\ldots,l_k | \hat c_j = l_j\hat x^{m_j} l_j^{-1} \mbox{, for } j=1,\ldots,k \rangle\]

Now by Proposition \ref{HNN}, $\hat w$ still generates a measure free factor in $\hat F$, and by Theorem \ref{lift}, we get that $w$ generates a measure free factor of $F$.
$\Box$

\begin{theorem} \label{twoletters} Let $G = F_2 = \langle a,b \rangle$. Then an element of the form $w=a^kb^n$ for $k,m \neq 0$ generates a measure free factor of $G$.
\end{theorem}

{\em Proof:}
It is enough to consider the case when $k,n>0$, for one can apply the automorphism taking $a$, $b$ to $a^{\mbox{sign}(k)}$, $b^{\mbox{sign}(n)}$.
Also, we can assume $k,n>1$, for if either one equals $1$, then $w$ is already a free factor.

Let $\Gamma$ be the rose on two petals, labeled by $a$ and $b$ respectively. So $\pi_1(\Gamma)=G$. Consider the graph $\hat \Gamma$ constructed as follows: It has $kn$ vertices, $v_j$ for $j\in\Z_{kn}$. As for the edges, for each $j\in\Z_{kn}$ there is an oriented edge labeled by $b$ going from $v_j$ to $v_{j+1}$, and an oriented edge labeled by $a$ going from $v_j$ to $v_{j+n}$. The orientations and labeling of the edges give a projection map $\hat \Gamma \to \Gamma$, which is a covering of index $kn$. Let $H = \pi_1(\hat \Gamma,v_0)$ be the corresponding subgroup of $G$.

\begin{figure}[htbp]
\centering
\includegraphics[width=0.75\textwidth]{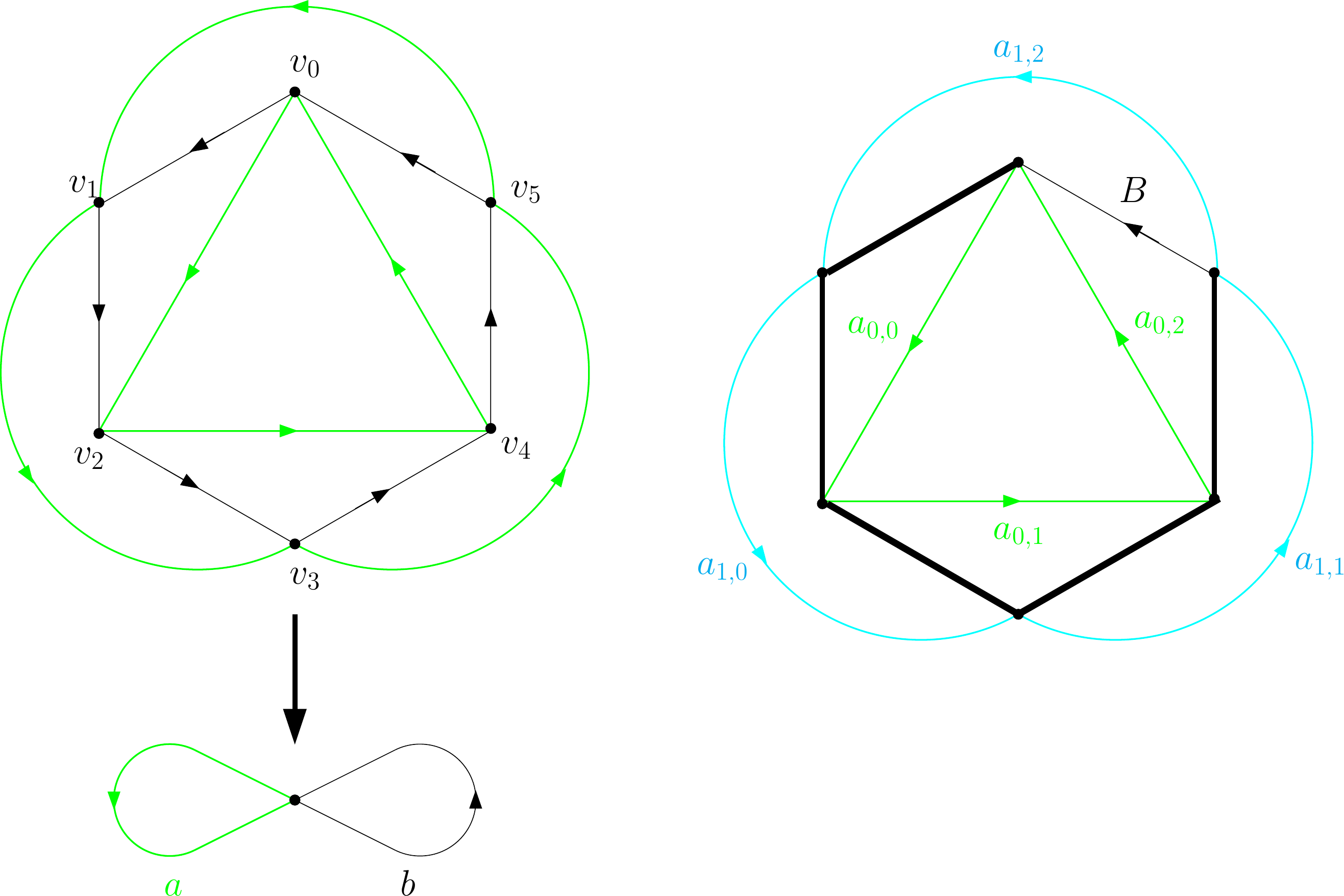}
\caption{{\bf Left:} An example of the cover $\hat \Gamma$ for $k=3,n=2$. {\bf Right:} The generators for $H$, where the bold edges form the spanning tree.}
\end{figure}

Consider the lift to $\hat \Gamma$ of a curve representing $w$ in $\Gamma$. If it starts at $v_j \in \hat \Gamma$, then it ends at $v_{j+n}$. This is because the lift of $a^k$ is closed, ending in $v_j$, and the lift of $b^n$ that starts at $v_j$ ends at $v_{j+n}$. Thus $w^k$ is the minimal power of $w$ whose lift from $v_j$ is a closed curve in $\hat \Gamma$. This holds true for any $j$, as the cover is normal.

Let $t=0,\ldots,n-1$ be a set of representatives of the cosets of $n\Z_{kn}\cong \Z_k$ in $\Z_{kn}$. Let $\gamma_j$ be the lift of $w^k$ starting from $v_t$, and let $w_t = [\beta_t \gamma_t \beta_t^{-1}]$ where $\beta_t$ is the lift of $b^t$ starting at $v_0$. Then $w_0,\ldots,w_{n-1}$ are a complete lift of $w$ to $H$.  

Now we will write the $w_t$ in a suitable basis for $H$.
Let $T$ be the spanning tree of $\hat \Gamma$ consisting of all the edges labeled by $b$ except $(v_{kn-1},v_{kn})$. Then the edges not in $T$ correspond to a free basis of $H$. Name them as follows: $B$ will stand for the generator corresponding to $(v_{kn-1},v_{kn})$, and $a_{t,i}$, for $t=0,\ldots,n-1$ and $i\in\Z_k$, will stand for the edge going from the vertex $v_{t+ni}$ to the vertex $v_{t+n(i+1)}$. This covers all edges not in $T$, observing that the edges labeled by $a$ are arranged in $n$ cycles of length $k$, each spanning a coset of $n\Z_{kn}$ in $\Z_{kn}$. 

Writing $w_t$ in this basis, we get \[w_t = (a_{t,0}\cdots a_{t,k-2} a_{t,k-1})(a_{t,1}\cdots a_{t,k-1} a_{t,0}) \cdots (a_{t,k-1} a_{t, 0}\cdots a_{t,k-2})B \]
Let $A_t = a_{t,0}\cdots a_{t,k-1}$. Then $w_t$ is the product of all cyclic conjugates of $A_t$ and $B$. 

Let $y_{t,i} = (a_{t,0}\cdot a_{i-1})^{-1}$. Then $B$, $A_t$, $y_{t,i}$, for $t=0,\ldots,n-1$ and $i=0,\ldots,k-2$ also forms a free basis of $H$. In this new basis, $w_t$ reads as
\[ w_t = A_t x_{t,0}A_t x_{t,0}^{-1}\cdots x_{t,k-2} A_t x_{t,k-2}^{-1} B \]
Using the previous result \ref{bswords}, we know that $v_t=A_t x_{t,0}A_t x_{t,0}^{-1}\cdots x_{t,k-2} A_t x_{t,k-2}^{-1}$ is a measure free factor of $H_t = \langle A_t,x_{t,i} \mbox{\ for } i=0,\ldots,k-2\rangle = \langle a_{t,i} \mbox{\ for }i=0,\ldots k-1 \rangle$. Now 
\[ H = H_0*\cdots *H_{n-1}*\langle B\rangle \]
So by Lemma \ref{freeprod} the subgroup $M=\langle v_0,\ldots,v_{k-1},B\rangle$ is a measure free factor of $H$, and the given basis generates it freely. Since $w_t = v_t B$, then $w_0,\ldots w_{k-1}, B$ is also a free basis of $M$. So $K=\langle w_0,\ldots, w_{n-1} \rangle$ is a measure free factor of $H$, freely generated by a complete lift of $w$. Theorem \ref{lift} finishes the proof.
$\Box$

Considering the conjugates of the words of the form $a^kb^n$ gives the following.

\begin{corollary} An element of the form $w=a^kb^na^p$ with $k \neq -p$, $n\neq 0$, generates a measure free factor of $F_2=\langle a,b\rangle$.
\end{corollary}

Observe that these are exactly all the $3$--letter words (i.e. of the form $a^kb^na^p$) that are not a proper power. 

\section{Measure free factors of virtually free groups} \label{sec:mff-vfree}

Finally, we consider measure free factors of the virtually free groups that are free products of free groups and finite cyclic groups.

\begin{theorem} \label{vfree} Let $G = \langle a_1,\ldots,a_n,s_1,\ldots,s_k | s_1^{n_1}=1,\ldots, s_k^{n_k}=1 \rangle \cong F_n*\Z_{n_1}*\cdots *\Z_{n_k}$. If $v\in F_n$ generates a measure free factor of $F_n$, then $w=vs_1^{p_1}\cdots s_k^{p_k}$ generates a measure free factor of $G$ for any $p_1,\ldots,p_k$.
\end{theorem}

\begin{proof}
We use induction on $k$. The base case $k=0$ is trivial. For the inductive step, consider the subgroup $K$ generated by $a_1,\ldots,a_n,s_1,\ldots,s_{k-1}$, and the subgroup $H$ generated by the conjugates $K_j=s_k^jKs_k^{-j}$ for $j=0,\ldots,n_k-1$. The natural presentation of $K$ is analogous to the one of $G$, only with $k-1$ generators of torsion. On the other hand, we shall see that $H\cong K_0*\cdots * K_{n_k-1}$ and has index $n_k$ in $G$. 

To prove this we consider the complex $X$ corresponding to the given presentation of $G$. We use the orbifold notation: $X$ consists on a wedge of circles in correspondence with the generators $a_1,\ldots,a_n, s_1,\ldots,s_k$, where the circle for $s_j$ is capped by a disk $D_j$ with a cone-point of degree $n_j$. Then $H$ corresponds to the $n_k$-sheeted branched cover $\hat X \to X$ constructed as follows: Let $\hat D_k \to D_k$ be the branched covering given by the map $\D\to\D / z\to z^{n_k}$ (identifying $D_k$ and $\hat D_k$ with the unit disk $\D\subset \C$, with the cone-point at $0$). Let $x_0,\ldots,x_{n_k-1}$ be the preimages of the basepoint of $X$ in $\hat D_k$, and $X_0,\ldots,X_{n_k-1}$ be copies of the presentation complex of $K$. We get $\hat X$ by wedging each $X_j$ to $\hat D_k$ at $x_j$, and the covering map $\hat X \to X$ is the natural extension of the branched cover $\hat D_k \to D_k$.

Notice that the single preimage in $\hat X$ of the cone-point of $D_k$ has degree $1$, i.e. is no longer a cone-point. So $\pi_1(\hat X) \cong \pi_1(X_0) * \cdots * \pi_1(X_{n_k-1})$, giving that $H=\pi_(\hat X) \cong K_0*\cdots *K_{n_k-1}$ and has index $n_k$ in $G$.

\begin{figure}[htbp]
\centering
\includegraphics[width=0.75\textwidth]{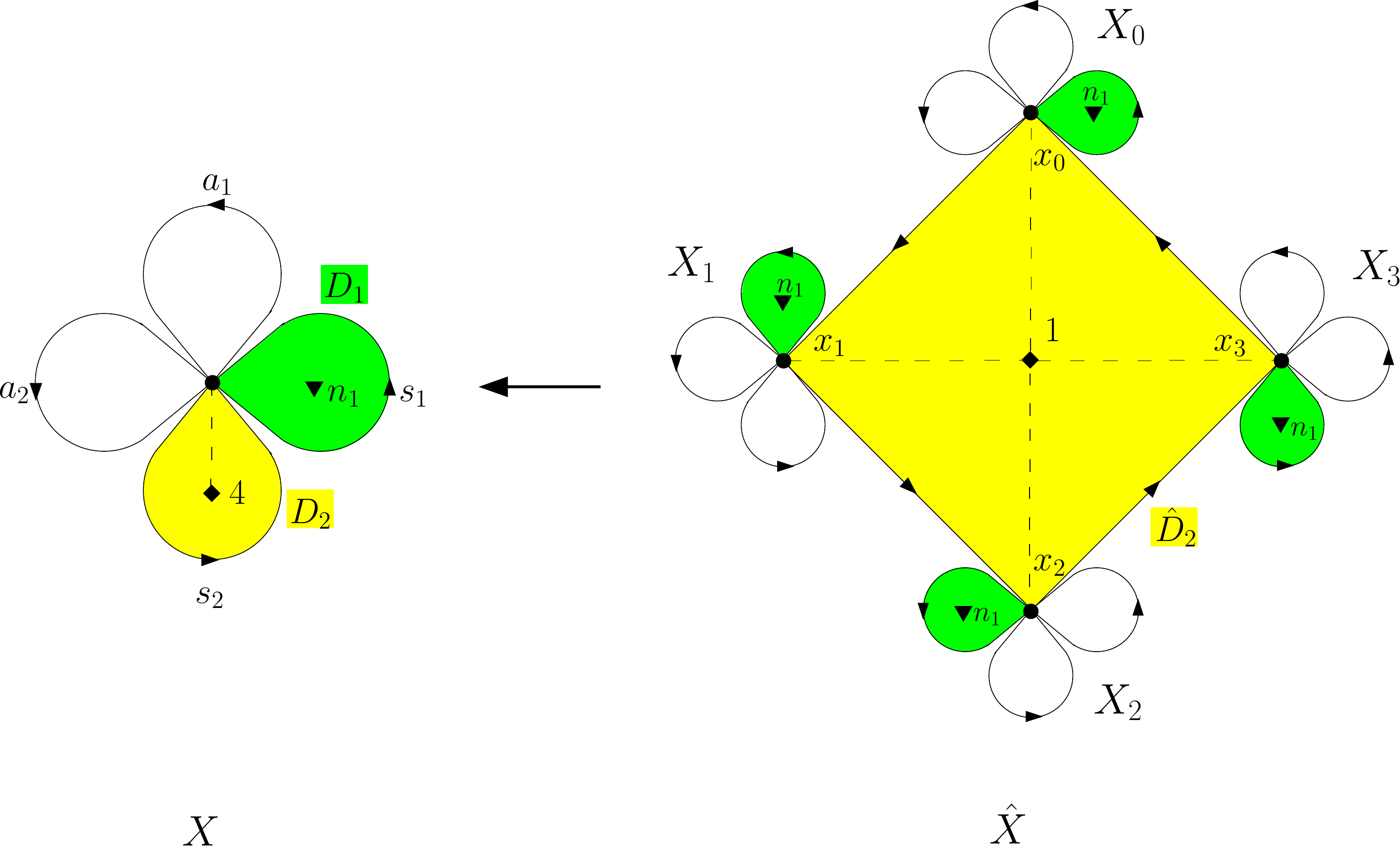}
\caption{{\bf Left:} The complex $X$ for $n=2$, $k=2$ and $n_2=4$. {\bf Right:} Sketch of the $4$-sheeted branched cover $\hat X$, and the subcomplexes $X_i$ used in the proof.}
\end{figure}

Now we will describe a complete lift for $w$. Let $u = vs_1^{p_1}\cdots s_{k-1}^{p_{k-1}} \in K$, and $u_j=s_k^jus_k^{-j} \in K_j$ for $j=0,\ldots,n_k-1$. Consider also $d=(p_k,n_k)$ and $m=n_k/d$. The lift of $w^m$ to $\hat X$ starting at $x_i$ is \[ w_i = u_{i} \cdots u_{i+(m-1)p_k} \] the product of the $u_{i+lp_k}$ for $l=0,\ldots,m-1$, in that order. Then $w_0,\ldots, w_{d-1}$ form a complete lift of $w$ to $H$.

For each $i=0,\ldots,d-1$ write $L_i = \ast_{l=0}^{m-1}K_{i+lp_k}$. By induction hypothesis, $u$ generates a measure free factor of $K$, and so does $u_j$ in $K_j$ by conjugation. Then the subgroup $M_i \leq L_i$, generated by $u_{i+lp_k}$ for $l=0,\ldots,m-1$, is a measure free factor of $L_i$, using Lemma \ref{freeprod}. Observe that $w_i$ is a free factor of $M_i$, and so it is a measure free factor of $L_i$ (this is an easy consequence of the definition of measure free factor). Finally, since $H = \ast_{i=0}^{d-1}L_i$, the subgroup generated by the complete lift $w_0,\ldots,w_{d-1}$ is free of rank $d$, and a measure free factor of $H$ (again by \ref{freeprod}). We finish by applying Theorem \ref{lift}.  

\end{proof}

This allows us to extend the results in section \ref{sec:surfaces} to the exact analogues for $2$--orbifold groups.

\begin{corollary} \label{orbifold} Let $S$ be a $2$--orbifold with boundary, and $\gamma_1,\ldots,\gamma_k$ be its boundary curves. Suppose that $\alpha = \{ \alpha_1,\ldots,\alpha_n \}$ is a family of disjoint essential simple closed curves on $S$, and $S_1,\ldots,S_t$ are the components of $S$ cut along $\alpha$. If $S_j$ has genus at least $1$ for every $j$, then the subgroup \[ \langle [\gamma_1],\ldots,[\gamma_k],[\alpha_1],\ldots,[\alpha_n]\rangle \leq \pi_1(S) \] is a measure free factor of $\pi_1(S)$. 
\end{corollary}

\begin{proof}

The arguments used in the proof of \ref{surface3} and \ref{surface4} allows us to reduce the corollary to the case with only one boundary component ($k=1$) and no curves $\alpha_i$ ($n=0$). Then we can write $\pi_1(S)$ as either 
\[ \pi_1(S) = \langle a_1,b_1,\ldots,a_g,b_g,s_1,\ldots,s_m | s_1^{n_1}=\cdots=s_m^{n_m} = 1\rangle \] where $[\gamma_1] = [a_1,b_1]\cdots[a_g,b_g]s_1\cdots s_m$, or
\[ \pi_1(S) = \langle a_1,\ldots,a_g,s_1,\ldots,s_m | s_1^{n_1}=\cdots=s_m^{n_m} = 1\rangle  \] where $[\gamma_1] = a_1^2\cdots a_g^2 s_1\cdots s_m$. In both cases we conclude using Proposition \ref{vfree}, together with Lemma \ref{surface3}.

\end{proof}

\end{document}